\documentclass[12pt]{amsart}
\usepackage{a4}
\usepackage{amssymb}
\usepackage{amsmath}
\usepackage{amsfonts}
\usepackage{fancybox}
\usepackage{empheq}
\usepackage{mathtools}
\usepackage{calrsfs}
\usepackage{changes}
\usepackage{framed}
\usepackage{color, colortbl}
\definecolor{LightCyan}{rgb}{0.88,1,1}
\definecolor{Gray}{gray}{0.9}

\definechangesauthor[color=orange,name={Aristides Kontogeorgis}]{AK}

\usepackage{color}
\usepackage{comment}
\usepackage[all,cmtip]{xy}

\usepackage{todonotes}

\usepackage{hyperref}
\usepackage{cleveref}

\usepackage{tikz}

\newtheorem{theorem}{Theorem}
\newtheorem{lemma}[theorem]{Lemma}

\newtheorem{proposition}[theorem]{Proposition}
\newtheorem{example}[theorem]{Example}
\newtheorem{remark}[theorem]{Remark}
\newtheorem{definition}[theorem]{Definition}

\newcommand{\init}{{\rm in_\prec}}

\usepackage[normalem]{ulem}

\title{A generating set for the canonical ideal of HKG-curves}

\date{\today}

\author[A. Kontogeorgis]{Aristides Kontogeorgis}
\address{Department of Mathematics, National and Kapodistrian  University of Athens\\
Panepistimioupolis, 15784 Athens, Greece}
\email{kontogar@math.uoa.gr}
\urladdr{http://users.uoa.gr/~kontogar}

\author[I. Tsouknidas]{Ioannis Tsouknidas }
\address{Department of Mathematics, National and Kapodistrian University of Athens\\
Panepistimioupolis, 15784 Athens, Greece}
\email{iotsouknidas@math.uoa.gr}

\date \today

\makeatletter
\newcommand{\aprod}{\mathop{\operator@font \hbox{\Large$\ast$}}}
\makeatother

\begin{document}
\begin{abstract}
The canonical ideal for Harbater Katz Gabber covers satisfying the conditions of Petri's theorem is studied
and an explicit non-singular model of the above curves is given. 
\end{abstract}

\maketitle


\section{Introduction}

The study of the canonical embedding and the determination of the canonical ideal is a classical subject in algebraic geometry, see \cite[III.3]{MR770932}, \cite{Saint-Donat73}, \cite[p. 20]{MR0419430}, \cite{MR943539} for a modern account.  On the other hand Harbater-Katz-Gabber curves (HKG-curves for short) 
grew out mainly due to work of Harbater \cite{harbater1980moduli} and of Katz and Gabber \cite{katz1986local}.  They are important because of the Harbater–Katz–Gabber compactification theorem of Galois actions on complete local rings and they 
proved to be an important tool in the study of local actions and the deformation theory of curves with automorphisms, see \cite{MR3651589}, \cite{MR2441248}, \cite{MR2919977}, \cite{Obus12}, \cite{ObusWewers}, \cite{MR3194816}, \cite{Karanikolopoulos2013-vg}, \cite{1901.08446}.

In \cite{1909.10282} we have studied the relation of the canonical ideal  
of a given curve and the action of the automorphism group on the space of holomorphic differentials.
It is expected that a lot of information of the deformation of the action is hidden in the canonical ideal, see also \cite{KaranProc}, \cite{1905.05545}. 

In this article we  will work over an 
 algebraically closed field  $k$  of characteristic $p>0$. 
Our aim is to calculate the canonical ideal of an HKG-curve $X/k$. 
In order to do so we  use a recent result by Haralampous et al. \cite{1905.05545} integrated here as proposition \ref{lemma 1} and  
 we employ the breakdown process of an HKG-curve into Artin-Schreier extensions as described in \cite{Karanikolopoulos2013-vg} and \cite{1901.08446}, 
while also expanding our understanding of the generating elements (section \ref{sec:preliminaries}). We will assume that the Galois group of the HKG-cover $X\rightarrow \mathbb{P}^1$ is a $p$-group. 
In this process we will use the symmetric  Weierstrass semigroup $H$ at the unique ramification point together with the explicit bases of polydifferentials based on the semigroup given in \cite[prop. 42]{Karanikolopoulos2013-vg}. 

We interpret the equations of the intermediate Artin-Schreier extensions as equations of quadratic differentials defining a set of relations $G_0$ and $G_{\bar{v},i}$, which we prove that are part of the canonical ideal, see prop. \ref{defofG0} and \ref{defofG1}. 
Notice that 
in order to be able to generate the canonical ideal by quadratic polynomials we have to assume that all intermediate extensions
 satisfy the assumptions of Petri's theorem, see lemma \ref{PetriCondNeed}.



In the last section (\ref{sec:Artin-Schreier}) we give several examples illustrating our construction. Those examples are used to facilitate the fact that, despite the possibly complicated definition of the generating sets (along with the proof), in specific computations they behave far more conveniently.

\section{Preliminaries}\label{sec:preliminaries}

A Harbater-Katz-Gabber cover (HKG-cover for short) is a Galois cover $X_{\mathrm{HKG}}\rightarrow \mathbb{P}^1$, such that there are at most two  branched $k$-rational points $P_1,P_2 \in \mathbb{P}^1$, where $P_1$ is tamely ramified and $P_2$ is totally and wildly  ramified. All other geometric points of $\mathbb{P}^1$ remain unramified. In this article we are mainly interested in $p$-groups so our HKG-covers have a unique ramified point
$P$, which is totally and wildly ramified.



On the other hand, the canonical ideal is described in terms of the  next theorem which is given following  Saint-Donat's \cite{Saint-Donat73} formulation:
\begin{theorem}[Max Noether-Enriques-Petri] \label{petristhm}
Under our initial assumptions (i.e. $X$ is complete non-singular non-hyperelliptic of genus $\geq 3$, $k$ is algebraically closed) then the following hold;
\begin{enumerate}
\item The canonical map 
\[\phi:\mathrm{Sym}(H^0(X,\Omega_X))\to \bigoplus_{n\geq 0}H^0(X,\Omega_X^n)\]
is surjective ($\mathrm{Sym}$ stands for the symmetric algebra).
\item The kernel of $\phi$, $I$, is generated by elements of degree $2$ and $3$.
\item $I$ is generated by elements of degree $2$ except in the following cases;
\begin{enumerate}
\item $X$ is trigonal 
\item $X$ is a plane quintic $(g=6)$.
\end{enumerate}
\end{enumerate}
\end{theorem}


\noindent
\begin{minipage}[l]{0.6\textwidth}
It is known  \cite{Karanikolopoulos2013-vg}, \cite{1901.08446} that an HKG curve is defined by a series of extensions
 $F_{i+1}=F_i(\bar{f}_i)$, 
where the irreducible polynomials of $\bar{f}_i$ are known to be of the form
\begin{equation}
\label{defeqi}
X^{p^{n_i}}+a^{(i)}_{n_i-1}X^{p^{n_i-1}}+\dots+a^{(i)}_{0} X- D_i
\end{equation}
for some element $D_i\in F_i$. 

The Weierstrass semigroup $H$ is generated by the elements
$\{|G_0|, \bar{m}_1,\dots, \bar{m}_k\}$
where $\bar{m}_i=p^{n_{i+1}+\dots+n_{k}}b_i$.
 Notice that the ramification groups are given by $|G_{b_{i+1}}|=p^{n_{i+1}+\dots+n_{k}}$
and they form the following filtration sequence
\[
G_0(P)=G_1(P)=\cdots=G_{b_1}(P)\gneqq 
G_{b_1+1}(P)=\cdots
\]
\[
\cdots=
G_{b_2}(P)\gneqq\cdots\gneqq G_{b_\mu}(P)\gneqq\{1\}.
\]
 We know that 
 $(b_i,p)=1$ and $|G_0|=p^{n_1+\dots+n_k}$, see \cite{Karanikolopoulos2013-vg}, \cite{1901.08446}. 
\end{minipage}
\begin{minipage}[r]{0.4\textwidth}
\[
\xymatrix{
  F=F_{k+1}=F_{k}(\bar{f}_{k}) \ar@{-}[d]^{p^{n_k}} 
\ar@{-}@/_5pc/_{|G_0|}[dddd]
  \\
  F_{k}=F_{k-1}(\bar{f}_{k-1}) \ar@{-}[d]^{p^{n_{k-1}}} \\
  F_{k-1}  \ar@{.}[d] \\
F_2=F_1(\bar{f}_1)\ar@{-}[d]^{p^{n_1}}\\
  F_1=k(\bar{f}_0) \ar@{=}[d] \\
  F_0=F^{G_1(P)}
}
\]
\end{minipage}
The above subset of the Weierstrass semigroup might not be the minimal set of generators, 
since 
this depends on whether $G_1(P)$ equals $G_2(P)$, see \cite[thm. 13]{Karanikolopoulos2013-vg}.
We will denote by 
\[
H_s=
\{
h: h \in H, h\leq s(2g-2)
\}
\]
the part of the  Weierstrass semigroup bounded by $s(2g-2)$.
We will also denote by $\mathbf{A}$ the set 
\begin{equation}
\label{Adef}
\mathbf{A}=\{(i_0,\ldots,i_k)  \in \mathbb{N}^{k+1}:  
i_0 |G_0|+\sum_{\nu=1}^k
i_\nu \bar{m}_\nu \leq 2g-2
\}.
\end{equation}
For each $h\in H_1$ we will select a fixed element $f_h$ with unique pole at $P$ of order $h$.
The sets $H_1$ and $\mathbf{A}$ have the same cardinality and moreover the map 
\begin{align}
\label{f-to-diff}
H_s \ni h 
&\longmapsto  
f_h df_0^{\otimes s},
\end{align}
gives rise to a basis of $H^0(X,\Omega^{s})$, see 
\cite[prop. 42]{Karanikolopoulos2013-vg}. This implies that the cardinality of $H_s$ is given by 
\[
\#H_s=
\begin{cases}
g & \text{ if } s = 1
\\
(2s-1)(g-1) 
& \text{ if } s>1.
\end{cases}
\]
We will denote by $\mathbb{T}^2$ 
 the 
monomials of $\mathrm{Sym}H^0(X,\Omega_X)$ of degree two (i.e. of the form $\omega_L\omega_K$). 
For a graded ring $S$ we will use $(S)_2$ to denote elements of degree $2$.

The information of the succesive extensions is encoded in the coefficients $a^{(i)}_{j}$ of the additive left part of eq. (\ref{defeqi}) \emph{and} in the elements $D_i\in F_{i-1}$. 
By  eq. (\ref{defeqi}), we know that the valuation of 
$D_{i}$ is $-p^{n_{i}}\bar{m}_{i}$. Since $D_{i}$ belongs to $F_{i}=F^{G_1(P)}(\bar{f}_1,\dots,\bar{f}_{i-1})$ and $F^{G_1(P)}=k(\bar{f}_0)$ (see \cite[remark 21]{Karanikolopoulos2013-vg}), one can express $D_{i}$ as 
\begin{equation}
\label{Ddef}
D_{i}(\bar{f}_0,\dots,\bar{f}_{i-1})
=
\sum_{ 
(\ell_0,\dots,\ell_{i-1}) \in \mathbb{N}^{i}
}
 a_{\ell_0,\dots,\ell_{i}}^{(i)}\bar{f}_0^{\ell_0}\dots \bar{f}_{i-1}^{\ell_{i-1}}.
\end{equation}
We will need the following
\begin{lemma}
\label{eqnorm}
Assume that $(\ell_0,\dots,\ell_{i-1}),(w_0,\dots, w_{i-1}) \in \mathbb{N}^{i}$ such that 
\begin{equation}\label{hypothesis}
1\leq \ell_\lambda, w_\lambda<p^{n_\lambda} \text{ for all } 1\leq \lambda \leq i-1
\end{equation}
and
\begin{equation}\label{eqlem}
\ell_0|G_0|+\ell_1\bar{m}_1+\dots+\ell_{i-1}\bar{m}_{i-1}=w_0|G_0|+w_1\bar{m}_1+\dots+w_{i-1}\bar{m}_{i-1}.
\end{equation}
Then $(\ell_0,\dots,\ell_{i-1})=(w_0,\dots, w_{i-1})$.
\end{lemma}
\begin{proof}
Assume that 
 $(\ell_0,\dots,\ell_{i-1})\neq(w_0,\dots, w_{i-1})$.
We have by assumption, 
after cancelling out $p^{n_i+\cdots+n_k}$ from both sides, 
\begin{equation}
\begin{split}
\ell_0p^{n_1+\dots+n_{i-1}}+\sum_{v=1}^{i-2}\ell_{v}&p^{n_{v+1}+\dots+n_{i-1}}b_{v}+\ell_{i-1}b_{i-1}=\\
&=w_0p^{n_1+\dots+n_{i-1}}+\sum_{v=1}^{i-2}w_{v}p^{n_{v+1}+\dots+n_{i-1}}b_{v}+w_{i-1}b_{i-1}.
\end{split}
\end{equation}
By the coprimality of $b_{i-1}$ and $p$ we get that $p^{n_{i-1}}$ divides $w_{i-1}-\ell_{i-1}$. Suppose that the last difference is not zero and assume without loss of generality  that it is positive i.e.
\[
w_{i-1}-\ell_{i-1}=\lambda p^{n_{i-1}},\, \lambda>0.
\]
Then $w_{i-1}$ is strictly greater than $p^{n_{i-1}}$ which contradicts the inequality (\ref{hypothesis}) so we must have $w_{i-1}=\ell_{i-1}$. Cancelling them out in eq. \ref{eqlem} allows us to perform the same procedure yielding $w_{i-2}=\ell_{i-2}$. 
After a finite number of repetitions we get $w_1=\ell_1$ which means that also $w_0$ equals $\ell_0$, a contradiction since the elements were assumed different.
\end{proof}

The following lemma allows us to manipulate these elements;
\begin{lemma}\label{mainlem}
Let $F=F_{k+1}$ be the top field, with generators $\bar{f}_i$, $i=0,\dots, k$ and associated irreducible polynomials $A_i$ as in equation (\ref{defeqi});
\[
A_i(X)=X^{p^{n_i}}+a_{n_i-1}^{(k)} X^{p^{n_i-1}}+\dots+a_0^{(k)} X- D_i,
\]
where  $D_i$ is given  in equation (\ref{Ddef}),

\[
D_{i}(\bar{f}_0,\dots,\bar{f}_{i-1})
=
\sum_{ 
(\ell_0,\dots,\ell_{i-1}) \in \mathbb{N}^{i}
}
 a_{\ell_0,\dots,\ell_{i}}^{(i)}\bar{f}_0^{\ell_0}\dots \bar{f}_{i-1}^{\ell_{i-1}}.
\]

 Then  one of the monomials $\bar{f}_0^{\ell_0}\dots \bar{f}_{i-1}^{\ell_{i-1}}$ has also pole divisor
$p^{n_i}\bar{m}_i P$  and this holds for all $i=1,\dots,k$.
\end{lemma}
\begin{proof}
Recall that $D_i\in F_i$, $\bar{f}_i\in F_{i+1}-F_i$ and the pole divisor of $D_i$ is $p^{n_i}\bar{m}_i P$. 
Suppose on the contrary (for $D_i$) that, none
of the monomial summands of $D_i$ has pole divisor of the desired order, $p^{n_i}\bar{m}_iP$. 
In other words,
\[\ell_0|G_0|+\ell_1\bar{m}_1+\dots+\ell_{i-1}\bar{m}_{i-1}\neq p^{n_i}\bar{m}_i\]
for all $\ell_0,\dots,\ell_{i-1}$ appearing as exponents. We can assume that   $\ell_\lambda, w_\lambda$ satisfy the inequality of eq. (\ref{hypothesis})
 for all exponents of all monomial summands of $D_i$ since, otherwise, we can substitute the corresponding element $\bar{f}_\lambda^{\ell_\lambda}$ with terms of smaller exponents because of its irreducible polynomial,
 see also eq. (\ref{defeqi}).

By the strict triangle inequality  there will be at least two different monomials 
$\bar{f}_0^{\ell_0}\dots \bar{f}_{i-1}^{\ell_{i-1}}$,
 $\bar{f}_0^{w_0}\dots \bar{f}_{i-1}^{w_{i-1}}$ in the sum of $D_i$ sharing the same valuation and the 
contradiction
 follows by lemma \ref{eqnorm}.
\end{proof}

%
%
%
%
%
%
\section{Preparation for the main theorem}
\label{sec:numberingbasiselements}
\noindent Define the Minkowski sum (recall the definition of $\mathbf{A}$ given in eq. (\ref{Adef}))
\[
\mathbf{A}+\mathbf{A}=\{L+K: L,K \in \mathbf{A}\},
\]
where $L+K=(i_0+j_0,\dots,i_k+j_k)$ for $L=(i_0,\dots,i_k)$, $K=(j_0,\dots, j_k)$.
There is a natural map 
\begin{equation}
\label{valMap}
\mathbb{N}^{k+1}  \ni (i_0,i_1,\ldots,i_k)=\bar{h}
 \longmapsto 
||\bar{h}||=i_0 |G_0|+\sum_{\nu=1}^k
i_\nu \bar{m}_\nu \in H_2 \subset \mathbb{N},
\end{equation}
which restricts to  the map
\begin{align}
\label{normMap}
\mathbf{A} + \mathbf{A} & 
\stackrel{||\cdot||}{\longrightarrow} H_2 \\
L+K & \longmapsto \begin{pmatrix}
L+K
\end{pmatrix} \begin{pmatrix}
|G_0|
\\
\nonumber
\bar{m}_1\\ \vdots\\ \bar{m}_k
\end{pmatrix}=(i_0+j_0)|G_0|+\sum_{v=1}^k(i_v+j_v)\bar{m}_v.
\end{align}

The map given in eq. (\ref{normMap}) is not one to one. 
In order to bypass that we introduce a suitable equivalence relation $\sim$ on $\mathbf{A}+\mathbf{A}$ so that there is a bijection 
\[
 \psi: \mathbf{A}+\mathbf{A}/\sim\longrightarrow H_2':=\mathrm{Im} \psi \subset H_2.
\]
\begin{definition}
\label{defofequivalencerelation}
Define the equivalence relation $\sim$ on $\mathbf{A}+\mathbf{A}$, by the rule
$$(L+K)\sim (L'+K') \text{ if and only if } 
||L+K||=||L'+K'||.
$$
\end{definition}
The function $\psi$ together with eq. (\ref{f-to-diff})
allows us to express a quadratic differential $\omega_h$ corresponding to an element $h \in H_2'$ as an element in $\mathbf{A}+\mathbf{A}$ by selecting a representative $L+K \in \mathbf{A}+\mathbf{A}$ of the class of $\psi(f)$. 
That is for every element $h\in H_2'$ we can write 
\begin{equation}
\label{sim-representative}
\psi([L_h+K_h])=h \text{ for certain elements } L_h,K_h\in \mathbf{A}.
\end{equation}

It is clear 
by our definitions
that the following equality holds.
\begin{equation}
\label{associationofbasisandequivalenceclasses}
  \left|\frac{\mathbf{A}+\mathbf{A}}{\sim}\right|=|H_2'| \leq |H_2|=3g-3.
\end{equation}
as we mentioned in the introduction, the reasons for the definition of the equivalence relation will be clear later but the curious reader may check proposition (\ref{lemma 1}).
 We will need the following:

\begin{lemma}\label{defofGamma}
The equivalence class of the element $L+K=(i_0+j_0,\ldots,i_k+j_k)\in \mathbf{A}+\mathbf{A}$ corresponds to the following set of degree $2$ monomials
\begin{equation*}
\Gamma_{L+K}=
\left\{\begin{aligned}&\omega_{A}\omega_{B}\in \mathrm{Sym }H^0(X,\Omega_X):
\text{ for }
A=(a_0,\ldots,a_k), B=(b_0,\ldots, b_k)\\
&  \text{such that:}\\
&((a_0+b_0)-(i_0+j_0))|G_0|+
\sum_{v=1}^{k-1}
\big(
a_v+b_v-(i_v+j_v)
\big)
\bar{m}_v=\lambda\bar{m}_kp^{n_k}\\
&\text{and }
a_k+b_k-(i_k+j_k)=-\lambda p^{n_k}
\text{ for some }\lambda\in \mathbb{Z}\end{aligned}\right\},
\end{equation*}
\end{lemma}
\begin{proof}
The equivalence class of $L+K$ is a subset of $\mathbf{A}+\mathbf{A}$ which corresponds to holomorphic 
differentials as described below:
Notice first that two equivalent elements $L+K$, $L'+K'$ satisfy 
\[
(i_0+j_0-(i_0'+j_0'))|G_0|+\sum_{v=1}^k(i_v+j_v-(i_v'+j_v))\bar{m}_v=0
\]
which, combined with the facts that $(\bar{m}_v,p)=1$ and $\bar{m}_i=p^{n_{i+1}+\dots+n_{k}}b_i$ yields that there is an integer $\lambda$ such that 
\begin{align}\label{equivalceconditions}
\big(
i_0+j_0-(i_0'-j_0')
\big)
p^{|G_0|-n_k}+\sum_{v=1}^{k-1}(i_v+j_v-(i_v'+k'_v))\frac{\bar{m}_v}{p^{n_k}}&=\lambda \bar{m}_k \\
\text{and }i_k'+j_k'-(i_k+j_k)&=\lambda p^{n_k}.
\end{align}
\end{proof}

\begin{remark}
By Petri's theorem  the canonical map $\phi$ (check eq. (\ref{petristhm})) 
 maps 
a degree $2$ polynomial  to $f_h df_0^{\otimes 2} \in H^0(X,\Omega_X^{\otimes 2})$, 
that is 
\begin{equation}
\label{Petri2express}
\phi
\left(
\sum_{\nu} a_\nu \omega_{L_\nu} \omega_{K_\nu}
\right)
= f_h df_0^{\otimes 2}, \qquad a_\nu \in k. 
\end{equation}
It is not correct that
a holomorphic $2$-differential
$f_h df_0^{\otimes 2}$ is the image of a single element $\omega_L \omega_K$.  
Indeed, for the genus $9$ Artin-Schreier curve 
\[
y^7-y=x^4
\]
a basis for the set of holomorphic differentials is 
derived by the set 
\[
\mathbf{A}=
\left\{\left[0, 0\right], \left[0, 1\right], \left[0, 2\right], \left[0, 3\right], \left[0, 4\right], \left[1, 0\right], \left[1, 1\right], \left[1, 2\right], \left[2, 0\right]\right\}
\]
\begin{align*}
\omega_{0,0}=x^{0}y^{0} dx,\quad
\omega_{0,1}=x^{0}y^{1} dx,\quad
\omega_{0,2}=x^{0}y^{2} dx,\quad
\omega_{0,3}=x^{0}y^{3} dx,\quad
\omega_{0,4}=x^{0}y^{4} dx,
\\
\omega_{1,0}=x^{1}y^{0} dx,\quad
\omega_{1,1}=x^{1}y^{1} dx,\quad
\omega_{1,2}=x^{1}y^{2} dx,\quad
\omega_{2,0}=x^{2}y^{0} dx
\quad\end{align*}
while the holomorphic $2$-differential $x^4 y dx^{\otimes 2}$ 
cannot
be expressed as a single monomial of the above differentials, but as the following linear combination
\[
 \omega_{0,4}^2-
 \omega_{0,2}^2
  =y(y^7-y) dx^{\otimes 2} =x^4 y dx^{\otimes 2}.
\]
If the $2$-differential 
$f_0^{i_0}\cdots f_k^{i_k} df_0^{\otimes 2}$ is the image of a single monomial $\omega_K \omega_L$ with $K+L=(i_0,\ldots,i_k)$, then it is clear that the element $h=|G_0|
i_0
+\sum_{\nu=1}^k \bar{m}_\nu i_\nu$ in $H_2$ is the image of $L+K \in \mathbf{A}+\mathbf{A}$. 
\end{remark}

\section{The generating sets of the canonical ideal}\label{genersetofcanonideal}

For any element $K=(i_0,\ldots,i_k)\in \mathbb{N}^{k+1}$ we will denote by $f_K$ the element $f_0^{i_0}\cdots f_k^{i_k}$.
\colorlet{shadecolor}{blue!7}
\begin{shaded}
\begin{proposition}
\label{defofG0}
Consider the sets of quadratic holomorphic differentials:
\begin{align*}
\label{G0def}
G_0 :=&
\{\omega_L\omega_K-\omega_{L'}\omega_{K'}\in \mathrm{Sym }H^0(X,\Omega_X): L+K=L'+K', L,K,L',K' \in \mathbf{A}\}.
\end{align*}
Then $G_0$ is contained in the canonical ideal.
\end{proposition}
\end{shaded}
\begin{proof}
For the canonical map
$\phi:\mathrm{Sym}(H^0(X,\Omega_X))\to \bigoplus_{n\geq 0}H^0(X,\Omega_X^n)$
one has;
\[
  \phi(\omega_K\omega_{L}-\omega_{K'}\omega_{L'})=f_{K+L}df_0^{\otimes2}-f_{K'+L'}df_0^{\otimes2}=0.
\]
\end{proof}
\begin{remark}
Since $G_0$ is included in the canonical ideal 
we have that 
\[
\omega_h^{\otimes 2}
=
\omega_{h'}^{\otimes 2}  
\]
modulo the canonical ideal for any selection of 
$K_h+L_h,K_{h'}+L_{h'}$ representing $h,h'\in \mathbf{A}+\mathbf{A}$ such that $K_h+L_h=K_{h'}+L_{h'}$.
 Therefore, we will denote 
$2$-differentials by $\omega^{\otimes 2}_h$. 
\end{remark}

Using this notation we can rewrite the summands of  $D_i$ in eq. (\ref{Ddef}) as 2-differentials as explained below; 

\begin{lemma}
\label{PetriCondNeed}
The elements $D_i\in F_i$ have degree less than $4g-4$, yielding that  $D_i\cdot df_0^{\otimes}$ are 2-holomorphic differentials in $F$. 
In particular every monomial summand
$\bar{f}_0^{\ell_0} \cdots \bar{f}_{i-1}^{\ell_{i-1}}$ that appears in the expression of $D_i$ given in eq. (\ref{Ddef}) can be given as an element
\[
(0,\ldots,0) + (\ell_0,\ldots,\ell_{i-1},0,\dots,0) \in \mathbf{A}+\mathbf{A}.  
\]
and the element  $D_i$ can be written as a $2$-differential as
\begin{equation}\label{defofDultimate}
  D_i \cdot df_0^{\otimes 2}=
 \sum_{
\substack
 {
 \bar{\lambda}=(\ell_0,\ldots,\ell_{i-1},0,\ldots,0)\in\mathbf{A}+\mathbf{A}\\
 ||\bar{\lambda}||\leq p^{n_i}\bar{m}_i}}
  a_{\bar{\lambda}}^{(i)}\omega_{\bar{\lambda}}^{\otimes 2},
\end{equation}

\end{lemma}
\begin{proof}
By equation (\ref{defeqi}) we have that 
the absolute value of the valuation of $D_i$ in $F_{i+1}$ is $p^{n_i}b_i$. We will first show that  $p^{n_i} b_i \leq 4g_{F_{i+1}}-4$. 

According to the Riemann-Hurwitz formula the genera of $F_{i+1}$ and $F_i$ are related by 
\begin{equation}
\label{RH}
2(g_{F_{i+1}}-1)=p^{n_i}2 (g_{F_{i}}-1)+ (b_i+1)(p^{n_i}-1).
\end{equation}
Therefore 
\begin{align}
4(g_{F_{i+1}}-1)-p^{n_i}b_i &= 2p^{n_i}2(g_{F_{i}}-1) + p^{n_i}b_i 
-2b_i + 2p^{n_i}-2
\nonumber
\\
&= 2p^{n_i}2(g_{F_{i}}-1) +(p^{n_i}-2) b_i +2(p^{n_i}-1).
\label{g-diff}
\end{align}
If $g_{F_{i}} \geq 1$ then we have the desired inequality. Suppose that $g_{F_{i}}=0$. This can only happen for $i=1$ since $p^{n_i}>1$ and $b_i>1$. Therefore we need to show that
\[
  b_1p^{n_1}-2p^{n_1}-2b_1-2\geq 0
\]
and we are working over the rational function field.
The assumption on our curve being non-hyperelliptic implies that $p^{n_i} >2$ as well as $b_i>2$ and the last inequality becomes
\begin{equation}\label{ineq_artinschreier}
b_i\geq\frac{2p^{n_i}+2}{p^{n_i}-2}
\end{equation}
which is satisfied for $p^n>7$. Also the remaining cases, i.e. $p^{n_i}=5, 7$ require $b_i$ to be $\geq 4$ which is also true since $b_i=2$ is exluded by non-hyperellipticity and $b_i=3$ by non-trigonality.

Now the rest can be proved by induction as follows; We showed that 
\begin{equation}\label{ineq_artinschreier2}
p^{n_i} b_i \leq 4g_{F_{i+1}}-4
\end{equation}

 When we move from $F_{i+1}$ to $F_{i+2}$ the absolute value of the valuation of $D_i$ becomes $p^{n_{i+1}+n_i}b_i$ and we need to show that 
\[
  p^{n_{i+1}+n_i}b_i\leq 4g_{F_{i+2}}-4
\]
By \ref{ineq_artinschreier2} we rewrite this as $p^{n_{i+1}}(4g_{F_{i+1}}-4)\leq 4g_{F_{i+2}}-4$ which by the Riemann-Hurwitz formula (stated above) is equivalent to $(b_{i+1}+1)(p^{n_{i+1}}-1)$ being non-negative, which holds.

\end{proof}

\begin{remark}
If we assume that $F_i$ is not trigonal nor hyperelliptic then the same holds for all fields $F_k$ for $k\geq i$, see \cite[Appendix]{Poonen-gonality}.
\end{remark}

 


The set $G_0$ does not contain all elements of the canonical ideal. For instance, it does not contain the information of the defining equation of the Artin-Schreier extension and also the canonical ideal is not 
expected to be 
binomial.

Before the definition of the other generating sets of the canonical ideal, let us provide some insight on the process used to construct the elements of these sets.

Equation (\ref{defeqi}) is satisfied by the element $\bar{f}_i$, i.e;
\[
\bar{f}_i^{p^{n_i}}+a^{(i)}_{n_i-1}\bar{f}_i^{p^{n_i-1}}+\dots+a^{(i)}_{0} \bar{f}_i- D_i=0.
\]
This equation can be multiplied by elements of the form
$\bar{f_0}^{v_0}\cdots\bar{f_k}^{v_k}$ for any $v_0,\dots,v_k$, giving rise to
\[
\bar{f_0}^{v_0}\cdots\bar{f_k}^{v_k}
\left(
\bar{f}_i^{p^{n_i}}+a^{(i)}_{n_i-1}\bar{f}_i^{p^{n_i-1}}+\dots+a^{(i)}_{0} \bar{f}_i- D_i
\right)=0,
\]
which equals
\[
  \bar{f_0}^{v_0}\cdots \bar{f}_i^{v_i+p^{n_i}}\cdots\bar{f_k}^{v_k}+\ldots+a^{(i)}_{0}\bar{f_0}^{v_0}\cdots \bar{f}_i^{v_i+1}\cdots\bar{f_k}^{v_k}-\bar{f_0}^{v_0}\cdots \bar{f}_i^{v_i}\cdots\bar{f_k}^{v_k}D_i
=0.
 \]
If the exponents $(v_0,\ldots,v_k)$ are selected so that each summand in the last equation is an element in
$\mathbf{A}+\mathbf{A}$, then the equation gives rise to an element in the canonical ideal.

\begin{shaded}
\begin{proposition}
\label{defofG1}
Set  
\begin{align*}
 \bar{v} &:=(v_0,\ldots,v_k)\in \mathbb{N}^{k+1} 
 \\
\bar{\gamma}_{\bar{v},i,\nu} &:=(v_0,\ldots,v_i+p^{n_i-\nu},v_{i+1},\ldots,v_k),\, 0\leq\nu\leq n_i\\
\text{such that } 
||\bar{\gamma}_{\bar{v},i,0}||&\leq 4g-4.
 \text{ Also set}\\
\Lambda_i &=
\{
\bar{\lambda}=(\ell_0,\ldots,\ell_{i-1})\in \mathbb{N}^{i}:  0\leq \ell_\nu < p^{n_\nu}
\text{ for } 1\leq \nu \leq i\}
 \\
\bar{\beta}_{\bar{v},i,\bar{\lambda}} &:=  
\iota_{i-1,k}(\bar{\lambda})+\bar{v}=
(\ell_0,\dots,\ell_{i-1}, 0, \ldots,0) +\bar{v}\in \mathbf{A}+\mathbf{A}.
\end{align*}
  Define
\begin{equation}\label{makrinar}
G_{\bar{v},i}:=\left\{
\omega_{\bar{\gamma}_{\bar{v},i,0}}^{\otimes 2}+ \sum_{\nu=1}^{n_i}
a_{\nu}^{(i)}\omega_{\bar{\gamma}_{\bar{v},i,\nu}}^{\otimes 2}
-
\sum_
{\substack{
\bar{\lambda}\in \Lambda_i \\
||\bar{\lambda}||\leq p^{n_i}\bar{m}_i
}}
 a_{\bar{\lambda}}^{(i)}
\omega_{\bar{\beta}_{\bar{v},i,\bar{\lambda}}}^{\otimes 2}
\right\}
\end{equation} 
Then $G_{\bar{v},i}$ is contained in the canonical ideal for $1\leq i\leq k$.
\end{proposition}
\end{shaded}
Notice here that $\bar{v}$ is fixed while $\bar{\lambda}$ 
is running.
\begin{proof}
Again consider
$\phi:\mathrm{Sym}(H^0(X,\Omega_X))\to \bigoplus_{n\geq 0}H^0(X,\Omega_X^n)$. Then;
\begin{align*}
&\phi\left(\omega_{\bar{\gamma}_{\bar{v},i,0}}^{\otimes 2}+ 
\sum_{\nu=1}^{i}
a_{\nu}^{(i)}\omega_{\bar{\gamma}_{\bar{v},i,\nu}}^{\otimes 2}
-
\sum_
{\substack{
\bar{\lambda}\in \Lambda_i \\
||\bar{\lambda}||\leq p^{n_i}\bar{m}_i
}}
 a_{{\bar{\lambda}}}^{(i)}
\omega_{\bar{\beta}_{\bar{v},i,\bar{\lambda}}}^{\otimes 2}\right)=\\
=&\left(f_{(v_0,\ldots,v_i+p^{n_i},\ldots,v_k)}
+
\sum_{\nu=1}^i a_{\nu}^{(i)}
f_{(v_0,\ldots,v_i+p^{n_i-\nu},\ldots,v_k)}
-\right.\\
&\left. -\sum_{\substack{
\bar{\lambda}\in \Lambda_i \\
||\bar{\lambda}||\leq p^{n_i}\bar{m}_i
}}
 a_{\bar{\lambda}}^{(i)}
f_{(\ell_0+v_0,\ldots,\ell_{i-1}+v_{i-1},v_i,\ldots,v_k)}\right)df_0^{\otimes2}=\\
=
&
f_{(v_0,\ldots,v_k)}
\left(
\bar{f}^{p^{n_i}}
+\sum_{\nu=1}^ia_v^{(i)}
\bar{f}^{p^{n_i-\nu}}
-
\sum_{\substack{
\bar{\lambda}\in \Lambda_i \\
||\bar{\lambda}||\leq p^{n_i}\bar{m}_i
}}
 a_{\bar{\lambda}}^{(i)}
f_{(\ell_0\dots+\ell_{i-1},0,\ldots,0)}\right)
df_0^{\otimes2},
\end{align*}
which equals $0$ due to the 
relation satisfied by the irreducible polynomial of $\bar{f}_i$.
\end{proof}

\section{The main theorem}\label{sec:maintheorem}

We define a term order which compares products of differentials as follows; Let $
\omega_{I_1} \omega_{I_2} \cdots
\omega_{I_d},
\omega_{I_1'} \omega_{I_2'} \cdots
\omega_{I_{d'}'}
$ be two such products and consider the $(k+1)-$tuples $I_1+\dots+I_d=(v_0,\dots,v_k)$, $I_1'+\dots+{I}_{d'}'=(v_0',\dots,v_k')$.

Set 
\[
\omega_{I_1} \omega_{I_2} \cdots\omega_{I_d}\prec\omega_{I_1'} \omega_{I_2'} \cdots
\omega_{I_{d'}'}\Leftrightarrow 
(v_0,\dots,v_k)<_{\mathrm{colex}} (v_0',\dots,v_k')
\]
that is 
\begin{itemize}
\item $v_k<v_k'$ or
\item $v_k=v_k'$ and $v_{k-1}<v_{k-1}'$
\item $\vdots$
\item $v_i=v_i'$ for all $i=k,\dots,1$ and $v_0<v_0'$
\end{itemize}
We are going to work with the initial terms of the sets defined in the last two propositions where, by ``initial term'' we mean a maximal term with respect to the colexicographical order. We denote initial terms with $\init(\cdot)$.

\begin{lemma}
\label{initG}
For the element $G_{\bar{v},i}$ of proposition \ref{defofG1}
we have that 
\[
\init(G_{\bar{v},i})=\omega_{\bar{\gamma}_{\bar{v},i,0}}.
\]
and also, in the polynomial  $G_{\bar{v},i}$ there is another summand which is smaller colexicographically than $\omega_{\bar{\gamma}_{\bar{v},i,0}}$ but has the same $||\cdot||$-value. 
\end{lemma}
\begin{proof}
Indeed, in eq. (\ref{makrinar}) there are two elements of maximal value in terms of $||\cdot||$. Namely 
$\omega_{\bar{\gamma}_{\bar{v},i,0}}$ and 
 $a_{\bar{\lambda}}^{(i)}
\omega_{\bar{\beta}_{\bar{v},i,\bar{\lambda}}}^{\otimes 2}$,
for the $\bar{\lambda}=(\ell_0,\dots,\ell_{i-1},0,\dots,0)\in\mathbf{A}+\mathbf{A}$ corresponding to the monomial $\bar{f}_0^{\ell_0}\cdots\bar{f}_{i-1}^{\ell_{i-1}}$ of minimum valuation
which exists due to lemma \ref{mainlem}.
From these two elements, $\omega_{\bar{\gamma}_{\bar{v},i,0}}$ is the biggest since it corresponds to the element
$(v_0,\ldots,v_i+p^{n_i},\ldots,v_k)$, while the other corresponds to the smaller element
$(v_0+l_0,\ldots,v_{i-1}+l_{i-1},v_i,\ldots,v_k)$, with respect to the colexicographical order. 
\end{proof}

\begin{shaded}
\begin{theorem}\label{mainth}
The canonical ideal is generated by  $G_0$ and by $G_{\bar{v},i}$, for $1\leq i \leq k$
 and for the $\bar{v}\in \mathbb{N}^{k+1}$ satisfying the inequality $||\bar{\gamma}_{\bar{v},i,0}||\leq 4g-4$.
\end{theorem}
\end{shaded}

\begin{remark}
In the above theorem the condition $||\bar{\gamma}_{\bar{v},i,0}||\leq 4g-4$ implies the condition 
$||\bar{\gamma}_{\bar{v},i,\nu}||\leq 4g-4$ for $0\leq \nu \leq n_i$.
 We will prove in lemma \ref{conditions} that it also implies the condition $||\bar{\beta}_{\bar{v},i,\bar{\lambda}}||\leq 4g-4$. This means that, the condition $||\bar{\gamma}_{\bar{v},i,\nu}||\leq 4g-4$ for $0\leq \nu \leq n_i$
   guarantees 
that, in $G_{\bar{v},i}$, not only the first term (i.e. $\omega_{\bar{\gamma},i,0}^{\otimes2}$), but also all the others correspond to $2$-differentials.
\end{remark}

\begin{lemma}
\label{conditions}
The condition $||\bar{\gamma}_{\bar{v},i,0}||\leq 4g-4$, or in other words;
\begin{equation}
\label{cond1}
v_0|G_0|+\sum_{\nu=1}^{k}v_\nu \bar{m}_\nu + p^{n_i} \bar{m}_i
\leq 4g-4
\end{equation} 
implies that $\bar{\beta}_{\bar{v},i,\bar{\lambda}}$  lies in $\mathbf{A}+\mathbf{A}$, that is, it is also a 2-differential, for all $\bar{\lambda}$ associated with the monomials of $D_i$.
\end{lemma}
\begin{proof}
For $\bar{\lambda}\in \Lambda_i$ let
\[
	\bar{\beta}_{\bar{v},i,\bar{\lambda}}=(v_0+\ell_0,\dots,v_{i-1}+\ell_{i-1}, v_i,\ldots,v_k).
\]
We need to show that
\[
(v_0+\ell_0)|G_0|+\sum_{\nu=1}^{k} v_{\nu}\bar{m}_\nu+
\sum_{\nu=1}^{i-1} \ell_\nu \bar{m}_\nu
\leq 4g-4.
\]
By  (\ref{cond1}) we  need to show that
\[
\ell_0|G_0|+\sum_{\nu=1}^{i-1} \ell_{\nu}\bar{m}_\nu\leq p^{n_i}\bar{m}_i.
\]
Note that $\bar{\lambda}$ is the exponents of a monomial 
summand
 of $D_i$ and, by the valuation's strict triangle inequality one has;
\begin{align*}
  v(f_{\bar{\lambda}}
  )&\geq v(D_i)\Leftrightarrow\\
	-(\ell_0|G_0|+\sum_{\nu=1}^{i-1} \ell_{\nu}\bar{m}_\nu)&\geq -p^{n_i}\bar{m}_i
\end{align*}
as expected,
where $f_{\bar{\lambda}}$ is $\bar{f}_0^{\ell_0}\cdots \bar{f}_{i-1}^{\ell_{i-1}}$
\end{proof}


\begin{definition}
Define $J$ to be the 
set of elements in the canonical ideal 
consisting of 
 the elements $G_0,G_{\bar{v},i}$ for $1\leq i \leq k$ and for the appropriate $\bar{v}\in \mathbb{N}^{k+1}$ satisfying the inequality $||\bar{\gamma}_{\bar{v},i,0}||\leq 4g-4$. 
\end{definition}
\noindent
\textbf{Outline of the proof.} 
In order to show that $\langle J \rangle$ is the canonical ideal we will use the following proposition, for  a proof see \cite{1905.05545}.
\begin{proposition}\label{lemma 1} 
Let $J$ be 
a set of homogeneous polynomials of degree $2$ containing  the elements $G_0$ and an extra set of generators $G'$ and let $I$ be the canonical ideal.  
Assume that the hypotheses imposed by Petri's theorem in order for the canonical ideal to be generated by polynomials of degree two are fulfilled.
 If 
$\dim_L \left(S/   
\langle \init J \rangle
\right)_2\leq 3(g-1)$, then $I= \langle J \rangle$.
\end{proposition}
In order to apply proposition \ref{lemma 1} we will show that
\begin{equation}\label{mainequality}
\left|\frac{\mathbf{A}+\mathbf{A}}{\sim}\right|
=
\dim 
\left(
\frac{S}{
 \langle \init (J)
\rangle}
\right)_2,
\end{equation}
where we already know, see eq. (\ref{associationofbasisandequivalenceclasses}), that the cardinality of the first quotient is   $\leq |H_2|=3g-3$.
We identify a $k$-basis of
 $\left(S/  \langle\init \langle J \rangle\right)_2$ with
$
\mathbb{T}^2-  
\{
 \init(f): f\in  J
 \}
$ 
 and, in order to prove equality (\ref{mainequality}), we define the map

\begin{align}\label{defofphi}
\Phi:\mathbb{T}^2- \{ \init(f): f\in  J \}
& \longrightarrow
 \frac{\mathbf{A}+\mathbf{A}}{\sim}
 \\
  \omega_{L}\omega_{K}\notag
&\longmapsto [L+K]
\end{align}




%
%

\begin{lemma}
\label{lessinAA}
If $(u_0,\ldots,u_k) \in \mathbf{A}+\mathbf{A}$ then every 
$(u_0',\ldots,u_k')$ with $0 \leq u_\nu' \leq u_\nu $ for $1\leq \nu \leq k$ is also in $\mathbf{A}+\mathbf{A}$. 
\end{lemma}
\begin{proof}
Since $\bar{u}=(u_0,\ldots,u_k) \in \mathbf{A}+\mathbf{A}$ there are $\bar{a}=(a_0,\ldots,a_k)$, $\bar{b}=(b_0,\ldots,b_k)$ with $\bar{u}=\bar{a}+\bar{b}$ and $\bar{a},\bar{b} \in \mathbf{A}$, that is 
$||\bar{a}||,||\bar{b}|| \leq 2g-2$. But then every $\bar{a}'$ (resp. $\bar{b}'$) with $\bar{a}'=(a_0',\ldots,a_k')$ (resp. $\bar{b}'=(b_0',\ldots,b_k')$) such that $0\leq a_\nu'  \leq a_\nu$ (resp. $0 \leq b_\nu' \leq b_\nu$) for $0 \leq \nu \leq k$ satisfies $||\bar{a}'|| \leq ||\bar{a}|| \leq 2g-2 $ (resp. $||\bar{b}'|| \leq ||\bar{b}|| \leq 2g-2$), that is 
$\bar{a}',\bar{b}' \in \mathbf{A}$. The result follows. 
\end{proof}

\noindent We start by showing that $\Phi$ is one-to-one.


\begin{lemma}\label{injectivityofphi}
$\Phi$ is 1-1.
\end{lemma}
\begin{proof}
Consider the following elements of $ \mathbf{A}$;
\begin{align*}
L &=(i_0,i_1,\ldots,i_\ell,\ldots,i_k) &
K &=(j_0,j_1,\ldots,j_\ell,\ldots,j_k) \\
L' &=(i_0',i_1',\ldots,i_\ell',\ldots,i_k') &
K' &=(j_0',j_1',\ldots,j_\ell',\ldots,j_k') \\
\end{align*}
such that, 

$\omega_K \omega_L, \omega_{L'}\omega_{K'}$ are in $\mathbb{T}^2-\{\mathrm{in}(f): f\in J\}$. Assume that 
 $\Phi(\omega_L\omega_K)=\Phi(w_L'w_K')$, i.e. $L+K\sim L'+K'$. 
 Suppose that $i_k+j_k=i_k'+j_k'$. Then we have the following equality;
\[
	(i_0+j_0)|G_0|+\sum_{\ell=1}^k(i_\ell+j_\ell)\bar{m}_\ell=
	(i_0'+j_0')|G_0|+\sum_{\ell=1}^k(i_\ell'+j_\ell')\bar{m}_\ell
\]
from which we cancel out the last terms and divide by $p^{n_k}$ in order to have
\[
	(i_0+j_0)p^{n_1+\dots+n_{k-1}}+\sum_{\ell=1}^{k-1}(i_\ell+j_\ell)\frac{\bar{m}_\ell}{p^{n_k}}=
	(i_0'+j_0')p^{n_1+\dots+n_{k-1}}+\sum_{\ell=1}^{k-1}(i_\ell'+j_\ell')\frac{\bar{m}_\ell}{p^{n_k}}.
\]



By repeating the above process we can assume that 
there is an  $\ell \leq k$ such that
$i_\nu'+j_\nu'=i_\nu+j_\nu$ for $\ell < \nu \leq k$ and 
$i_\ell'+j_\ell' \neq i_\ell+j_\ell$ and assume without loss of generality that 
 $i_\ell'+j_\ell'>i_\ell+j_\ell$. Then by lemma (\ref{defofGamma}), we would have
\begin{equation}\label{positivelambdaequality}
i_\ell'+j_\ell'-(i_\ell+j_\ell)=\lambda p^{n_\ell}
\end{equation}
for $\lambda>0$. Using this we will show that $\omega_{L'}\omega_{K'}$ belongs to $  \init (J)$. In order to do that, we need to build an element $G_{i,\bar{v}}$ which has $\omega_{L'}\omega_{K'}$ as its initial term. In other words we look for an element of the following form;
\begin{equation}\label{redEQ}
\omega_{\bar{\gamma}_{\bar{v},i,0}}^{\otimes 2}+ \sum_{\nu=1}^{n_i}
a_{\nu}^{(i)}\omega_{\bar{\gamma}_{\bar{v},i,\nu}}^{\otimes 2}
-
\sum_
{\substack{
\lambda\in \Lambda_i \\
||\bar{\lambda}||\leq p^{n_i}\bar{m}_i
}}
 a_{\bar{h}}^{(i)}
\omega_{\bar{\beta}_{\bar{v},i,\bar{\lambda}}}^{\otimes 2},
\end{equation}
where $\omega_{\bar{\gamma}_{\bar{v},i,0}}^{\otimes 2}=\omega_{L'+K'}^{\otimes 2}$
and everything else should be as defined in proposition (\ref{defofG1}). This comes down to finding $\bar{v}=(v_0,\ldots,v_k)\in \mathbb{N}^{k+1}$ such that 
\[
   (v_0,\ldots,v_\ell+p^{n_\ell},
   v_{\ell+1}\ldots,
   v_{k}) =
   (v_0,\ldots,v_\ell+p^{n_\ell},
   i'_{\ell+1}+j'_{\ell+1} \ldots,
   i'_{k}+j'_k)
   =L'+K'.
\]
Indeed, recall that if we match our element with an initial term corresponding to $\bar{f}_\ell^{p^{n_\ell}}$ then all the other terms can be defined by the equation of the irreducible polynomial of $\bar{f}_\ell$.

Define $\bar{v}$ as follows;
\[
  v_s=
  \begin{cases}
  i_s'+j_s'&\text{for }s\neq\ell\\
  i_{\ell}'+j_{\ell}'-p^{n_{\ell}} &\text{for }s=l
  \end{cases}
\]

The element $(v_0,\dots,v_k)$ lies in $\mathbf{A}+\mathbf{A}$.
Indeed, since $L'+K'$ is in $\mathbf{A}+\mathbf{A}$, according to lemma \ref{lessinAA} we only need to show that $0 \leq v_\nu$ for all $0 \leq \nu \leq k$. The only thing that needs to be checked is whether $v_\ell$ is nonnegative. Equivalently, whether $u_\ell\geq p^{n_\ell}$. Now recall that 
$u_\ell=i_{\ell}'+j_{\ell}'=\lambda p^{n_\ell}+(i_\ell+j_\ell)$
$v_\ell=i_\ell'+j_\ell'-p^{n_\ell}=i_\ell+j_\ell+(\lambda-1)p^{n_\ell}$
 by eq. (\ref{positivelambdaequality}), and $\lambda\geq 1$ hence
\[
  \lambda p^{n_\ell}+(i_\ell+j_\ell)\geq p^{n_\ell}
\]
as expected. 

This proves that  $\omega_{L'+K'}$ is the initial term of $G_{\bar{v},\ell}$ 
for $\bar{v}=(v_0,\ldots,v_k)$,
check also lemma (\ref{initG}),  giving us 
a contradiction so the map $\Phi$ is 1-1.
\end{proof}

\begin{lemma}\label{surjectivityofphi}
The map $\Phi$
 is surjective.
\end{lemma}
\begin{proof}
Take an equivalence class $[L+K]$ in $(\mathbf{A}+\mathbf{A})/\!\!\sim$. Recall the definition of the set  $\Gamma_{L+K}$ given in lemma  \ref{defofGamma}.
Consider the minimal element of $\Gamma_{L+K}$ , i.e. $\mathrm{min}\Gamma_{L+K}:=\omega_A\omega_B\in\mathbb{T}^2$. There is
 such 
 a minimal element since $\Gamma_{L+K}$ is nonempty (for example $\omega_L\omega_K\in \Gamma_{L+K}$) and since our order is a total order. We still need to show that $\omega_A\omega_B$ is not
  in  $  \init (J) $.


Firstly suppose that $\omega_A\omega_B\in\init(G_0)$. Then there is $\omega_I\omega_J$ such that $\omega_I\omega_J\prec\omega_A\omega_B$ and $A+B=I+J$.  By the last equality, $||A+B||=||I+J||$ so $A+B\sim I+J$. But this means that $\omega_I\omega_J$ is also in $\Gamma_{L+K}$ and is 
colexicographically
smaller than $\omega_A\omega_B$, a contradiction.

Suppose now that $\omega_A\omega_B \in \init(G_{\bar{v},i})$ for some $\bar{v},i$. Then according to lemma \ref{initG} there is a second element in the polynomial $G_{\bar{v},i}$ which has the same value when  $||\cdot||$ is applied,  but is smaller in $\prec$ (a contradiction since, having the same $||\cdot||-$value means that they are equivalent i.e. they both lie in $\Gamma_{L+K}$).
\end{proof}
\section{Examples}
We provide here some explicit examples of our method for calculating the canonical ideal of HKG curves.

\subsection{Artin-Schreier curves}
\label{sec:Artin-Schreier}Here we write down the generating sets of the canonical ideal corresponding to Artin-Schreier curves of the form  
\begin{equation}
\label{ASclassic}
X: y^{p^n}-y=x^m, \qquad (m,p)=1,
\end{equation}
where the values of  $m,p$ are given in the following table. 
	Notice that these curves form an example of an HKG-cover extension for the $k=1$ case. 
\[
\begin{array}{|c|c|}
\hline
m & \text{Petri's theorem requirement} \\
\hline
m > 5 & p^n > 3 
\\
m=4,5 & p^n \geq 5 \\
\hline
\end{array}
\] 
In this case the genus $g$ of the curve is $g>6$ and also the curve is not hyperelliptic nor trigonal.
Indeed the  above given curves have Weierstrass semigroup
 \begin{equation}
 \label{HAS}
 H:=m \mathbb{Z}_+ + p^n \mathbb{Z}_+
 \end{equation}
at the unique ramified point $P$. 
Let $G$ be the $p^n$ order Artin-Schreier cover group generated by the automorphism $\tau: y\mapsto y+1, x\mapsto x$.
Assume that there is a degree two covering 
$X \rightarrow \mathbb{P}^1$. This is a Galois covering with Galois group generated by the hyperelliptic involution $j:X \rightarrow X$. The hyperelliptic involution cannot be in the $p$-order Galois group $G$ of the 
Artin-Schreier extension, since $p$ is odd. On the other hand it is well known that the hyperelliptic involution is in the center of the automorphism group of $X$, \cite{BraStich86}. 
Since  $\tau(j(P))=j \tau(P)=P$ we have $j(P)=P$, otherwise the Galois cover $X\rightarrow X/G=\mathbb{P}^1$ has two ramified points, a contradiction. But then $2$ should be a pole number 
of the semigroup $H$, contradicting eq. (\ref{HAS}).

In order  prove that $X$ is 
also not trigonal,
 we can employ the fact that with the assumptions given in the table above we can indeed find a quadratic basis of the canonical ideal. Alternatively we can argue as follows: In characteristic zero we know that  at a non ramified point in the degree 3 cover $X\rightarrow \mathbb{P}^1$ 
of a trigonal curve either $3n$ or $2n$ is a pole number
for $(g-1)/n \leq n \leq g/2$, see \cite{MR1043748}. 
On the other hand for a Weierstrass point of the trigonal curve we have 
than every $k\geq s+2$ is a pole number, where $g-1 \leq s \leq 2g-2$. 
Lefschetz principle implies that this is the structure of Weierstrass semigroups for a big enough prime $p$. On the other hand, the  ramified point $P$ in the Artin-Schreier cover is a Weierstrass point, see \cite[th. 1]{garciaelab}. The semigroup strucure at $P$ given in eq. (\ref{HAS})
is not compatible with any of the Weierstrass semigroups of trigonal curves, therefore the curve $X$ is not trigonal at least for big enough $p$. 


Recall that $H_i$  denotes the bounded parts of the 
Weierstrass semigroup. 
For the case at hand we
 have that 
\begin{align*}
|H_1| & =g=(m-1)(p^n-1)/2 \\
|H_2| & =3(g-1).
\end{align*}
Also $\mathbf{A}=\{
L:=(i_0,i_1):  i_0 p^n +i_1 m \leq 2(g-1)
\}$
and
\[
\mathbf{A}+\mathbf{A}=\{L+K=(i_0+j_0,i_1+j_1) \mid L:=(i_0,i_1)\in \mathbf{A},\, K:=(j_0,j_1)\in \mathbf{A}\}.
\]

The equivalence class of  $L+K\in \mathbf{A}+\mathbf{A}$, as described in lemma (\ref{defofGamma}),corresponds to the following set of degree 2 monomials
\[
\Gamma_{L+K}=\{\omega_{A}\omega_{B}\in \mathrm{Sym }H^0(X,\Omega_X):A+B-(L+K)=(\lambda m,-\lambda p^n) \text{ for some }\lambda\in \mathbb{Z}\}.
\]
According to proposition \ref{defofG0} $G_0$ is defined by
\begin{equation*}
G_0 :=
\{\omega_L\omega_K-\omega_{L'}\omega_{K'}\in \mathrm{Sym }H^0(X,\Omega_X): L+K=L'+K', L,K,L',K' \in \mathbf{A}\}.
\end{equation*} 
The sets $G_{\bar{v},i}$ containing the information of the Artin-Schreier extension now adopt the following, much simpler form;
\[
  G_{(v_0,v_1),1}=\left\{
\omega_{(v_0,v_1+p^{n})}^{\otimes 2}-
\omega_{(v_0,v_1+1)}^{\otimes 2}
-
\omega_{(v_0+m,v_1)}^{\otimes 2}
\right\}
\]
for the $\bar{v}:=(v_0,v_1)$ satisfying $||(v_0,v_1+p^n)||\leq4g-4$, equivalently,
\[
  v_0p^n+v_1m+p^nm\leq 4g-4.
\]
Notice that if $p,n$ and $m$ are given specific values, the last inequality can be solved explicitly and the generating sets can be written down.

\begin{example}
Consider the Artin-Schreier curve $y^7-y=x^4$ of genus $9$. The canonical ideal is generated by 
the set $G_0$ given by 
\tiny{
\begin{multline*}
\{-w_{04} w_{10} + w_{03} w_{11}, -w_{10} w_{11} + w_{01} w_{20}, w_{04} w_{10} - w_{03} w_{11}, w_{10} w_{11} - w_{01} w_{20},
 -w_{02}^{2} + w_{01} w_{03}, w_{02}^{2} - w_{01} w_{03},
 \\
  -w_{01} w_{11} + w_{00} w_{12}, 
 w_{01} w_{11} - w_{00} w_{12},
 -w_{02} w_{11} + w_{01} w_{12}, w_{02} w_{11} - w_{01} w_{12}, -w_{11} w_{12} + w_{03} w_{20}, -w_{02}^{2} + w_{00} w_{04}, 
\\
 w_{02}^{2} - w_{00} w_{04}, w_{11} w_{12} - w_{03} w_{20},
  -w_{01}^{2} + w_{00} w_{02}, w_{01}^{2} - w_{00} w_{02}, -w_{03} w_{11} + w_{02} w_{12}, -w_{11}^{2} + w_{02} w_{20},
\\
   w_{03} w_{11} - w_{02} w_{12}, -w_{10} w_{12} + w_{02} w_{20}, w_{10} w_{12} - w_{02} w_{20}, w_{11}^{2} - w_{02} w_{20}, -w_{02} w_{10} + w_{00} w_{12}, w_{02} w_{10} - w_{00} w_{12},
\\
   w_{03} w_{10} - w_{02} w_{11}, -w_{03} w_{10} + w_{01} w_{12}, 
w_{03} w_{10} - w_{01} w_{12}, -w_{02} w_{10} + w_{01} w_{11}, -w_{10}^{2} + w_{00} w_{20}, w_{02} w_{10} - w_{01} w_{11}, 
\\
w_{10}^{2} - w_{00} w_{20},
 -w_{11}^{2} + w_{10} w_{12}, w_{11}^{2} - w_{10} w_{12}, -w_{01} w_{02} + w_{00} w_{03}, -w_{04} w_{11} + w_{03} w_{12}, w_{04} w_{11} - w_{03} w_{12}, 
\\
 -w_{02} w_{03} + w_{01} w_{04},
  w_{02} w_{03} - w_{01} w_{04}, -w_{03}^{2} + w_{02} w_{04}, w_{03}^{2} - w_{02} w_{04}, -w_{01} w_{03} + w_{00} w_{04}, w_{01} w_{03} - w_{00} w_{04},
\\
   -w_{04} w_{10} + w_{02} w_{12},
 w_{04} w_{10} - w_{02} w_{12}, -w_{01} w_{10} + w_{00} w_{11}, w_{01} w_{10} - w_{00} w_{11}, w_{12}^{2} - w_{04} w_{20},
\\
 -w_{12}^{2} + w_{04} w_{20}, w_{01} w_{02} - w_{00} w_{03}, -w_{03} w_{10} + w_{02} w_{11}
 \}
\end{multline*}
}
and one trinomial 
\[
-w_{00} w_{01} + w_{03} w_{04} - w_{20}^{2}
\]
\end{example}

 \subsection{HKG-covers with $p$-cyclic group}
This is a case where all the intermediate subextensions $F_i/F_{i-1}$ are of degree $p$ and the corresponding irreducible polynomials are
\[
  X^p+a^{(i)}X-D_i
\]
In this case the generating sets of the canonical ideal are
\begin{equation*}
G_0 :=
\{\omega_L\omega_K-\omega_{L'}\omega_{K'}\in \mathrm{Sym }H^0(X,\Omega_X): L+K=L'+K', L,K,L',K' \in \mathbf{A}\}
\end{equation*}

\begin{equation}
G_{\bar{v},i}:=\left\{
\omega_{(v_0,\dots,v_i+p,\dots,v_k)}^{\otimes 2}+
a^{(i)}\omega_{(v_0,\dots,v_i,\dots,v_k)}^{\otimes 2}
-
\sum_
{\substack{
\bar{\lambda}\in\mathbf{A}+\mathbf{A}\\
||\bar{\lambda}||\leq 
p \bar{m}_i 
}}
 a_{\bar{h}}^{(i)}
\omega_{\bar{\beta}_{\bar{v},i,\bar{\lambda}}}^{\otimes 2}
\right\}
\end{equation} 
such that
$||\bar{\gamma}_{\bar{v},i,0}||\leq 4g-4$
where $\bar{\beta}_{\bar{v},i,\bar{\lambda}} = 
(l_0,\dots,l_{i-1}, 0, \ldots,0) +\bar{v}$
as defined before.

 \def\cprime{$'$}


\begin{thebibliography}{10}

\bibitem{MR770932}
E.~Arbarello, M.~Cornalba, P.~A. Griffiths, and J.~Harris.
\newblock {\em Geometry of algebraic curves. {V}ol. {I}}, volume 267 of {\em
  Grundlehren der Mathematischen Wissenschaften [Fundamental Principles of
  Mathematical Sciences]}.
\newblock Springer-Verlag, New York, 1985.
\newblock \href {https://doi.org/10.1007/978-1-4757-5323-3}
  {\path{doi:10.1007/978-1-4757-5323-3}}.

\bibitem{MR3651589}
Frauke~M. Bleher, Ted Chinburg, Bjorn Poonen, and Peter Symonds.
\newblock Automorphisms of {H}arbater-{K}atz-{G}abber curves.
\newblock {\em Math. Ann.}, 368(1-2):811--836, 2017.
\newblock \href {https://doi.org/10.1007/s00208-016-1490-2}
  {\path{doi:10.1007/s00208-016-1490-2}}.

\bibitem{BraStich86}
Rolf Brandt and Henning Stichtenoth.
\newblock Die {A}utomorphismengruppen hyperelliptischer {K}urven.
\newblock {\em Manuscripta Math.}, 55(1):83--92, 1986.

\bibitem{1905.05545}
Hara Charalambous, Kostas Karagiannis, and Aristides Kontogeorgis.
\newblock The relative canonical ideal of the
  {A}rtin-{S}chreier-{K}ummer-{W}itt family of curves, 2019.
\newblock \href {http://arxiv.org/abs/arXiv:1905.05545}
  {\path{arXiv:arXiv:1905.05545}}.

\bibitem{MR2441248}
T.~Chinburg, R.~Guralnick, and D.~Harbater.
\newblock Oort groups and lifting problems.
\newblock {\em Compos. Math.}, 144(4):849--866, 2008.
\newblock \href {https://doi.org/10.1112/S0010437X08003515}
  {\path{doi:10.1112/S0010437X08003515}}.

\bibitem{MR2919977}
Ted Chinburg, Robert Guralnick, and David Harbater.
\newblock The local lifting problem for actions of finite groups on curves.
\newblock {\em Ann. Sci. \'{E}c. Norm. Sup\'{e}r. (4)}, 44(4):537--605, 2011.
\newblock \href {https://doi.org/10.24033/asens.2150}
  {\path{doi:10.24033/asens.2150}}.

\bibitem{garciaelab}
Arnaldo Garc{\'{\i}}a.
\newblock On {W}eierstrass points on certain elementary abelian extensions of
  {$k(x)$}.
\newblock {\em Comm. Algebra}, 17(12):3025--3032, 1989.

\bibitem{harbater1980moduli}
David Harbater.
\newblock Moduli of p-covers of curves.
\newblock {\em Communications in Algebra}, 8(12):1095--1122, 1980.

\bibitem{KaranProc}
Sotiris Karanikolopoulos and Aristides Kontogeorgis.
\newblock Integral representations of cyclic groups acting on relative
  holomorphic differentials of deformations of curves with automorphisms.
\newblock {\em Proc. Amer. Math. Soc.}, 142(7):2369--2383, 2014.
\newblock URL: \url{https://doi.org/10.1090/S0002-9939-2014-12010-7}.

\bibitem{Karanikolopoulos2013-vg}
Sotiris Karanikolopoulos and Aristides Kontogeorgis.
\newblock Automorphisms of curves and {W}eierstrass semigroups for
  {H}arbater-{K}atz-{G}abber covers.
\newblock {\em Trans. Amer. Math. Soc.}, 371(9):6377--6402, 2019.
\newblock \href {https://doi.org/10.1090/tran/7562}
  {\path{doi:10.1090/tran/7562}}.

\bibitem{katz1986local}
Nicholas~M Katz.
\newblock Local-to-global extensions of representations of fundamental groups.
\newblock {\em Ann. Inst. Fourier (Grenoble)}, 36(4):69--106, 1986.

\bibitem{MR1043748}
Seon~Jeong Kim.
\newblock On the existence of {W}eierstrass gap sequences on trigonal curves.
\newblock {\em J. Pure Appl. Algebra}, 63(2):171--180, 1990.
\newblock \href {https://doi.org/10.1016/0022-4049(90)90024-C}
  {\path{doi:10.1016/0022-4049(90)90024-C}}.

\bibitem{1909.10282}
Aristides Kontogeorgis, Alexios Terezakis, and Ioannis Tsouknidas.
\newblock Automorphisms and the canonical ideal, 2019.
\newblock \href {http://arxiv.org/abs/arXiv:1909.10282}
  {\path{arXiv:arXiv:1909.10282}}.

\bibitem{1901.08446}
Aristides Kontogeorgis and Ioannis Tsouknidas.
\newblock A cohomological treatise of {HKG}-covers with applications to the
  {N}ottingham group.
\newblock {\em J. Algebra}, 555:325--345, 2020.
\newblock \href {https://doi.org/10.1016/j.jalgebra.2020.02.037}
  {\path{doi:10.1016/j.jalgebra.2020.02.037}}.

\bibitem{MR0419430}
David Mumford.
\newblock {\em Curves and their {J}acobians}.
\newblock The University of Michigan Press, Ann Arbor, Mich., 1975.

\bibitem{Obus12}
Andrew Obus.
\newblock The (local) lifting problem for curves.
\newblock In {\em Galois-{T}eichm\"uller theory and arithmetic geometry},
  volume~63 of {\em Adv. Stud. Pure Math.}, pages 359--412. Math. Soc. Japan,
  Tokyo, 2012.

\bibitem{ObusWewers}
Andrew Obus and Stefan Wewers.
\newblock Cyclic extensions and the local lifting problem.
\newblock {\em Ann. of Math. (2)}, 180(1):233--284, 2014.
\newblock URL: \url{https://doi.org/10.4007/annals.2014.180.1.5}.

\bibitem{Poonen-gonality}
Bjorn Poonen.
\newblock Gonality of modular curves in characteristic {$p$}.
\newblock {\em Math. Res. Lett.}, 14(4):691--701, 2007.
\newblock \href {https://doi.org/10.4310/MRL.2007.v14.n4.a14}
  {\path{doi:10.4310/MRL.2007.v14.n4.a14}}.

\bibitem{MR3194816}
Florian Pop.
\newblock The {O}ort conjecture on lifting covers of curves.
\newblock {\em Ann. of Math. (2)}, 180(1):285--322, 2014.
\newblock \href {https://doi.org/10.4007/annals.2014.180.1.6}
  {\path{doi:10.4007/annals.2014.180.1.6}}.

\bibitem{Saint-Donat73}
B.~Saint-Donat.
\newblock On {P}etri's analysis of the linear system of quadrics through a
  canonical curve.
\newblock {\em Math. Ann.}, 206:157--175, 1973.
\newblock URL: \url{https://doi.org/10.1007/BF01430982}.

\bibitem{MR943539}
Karl-Otto St\"{o}hr and Paulo Viana.
\newblock A variant of {P}etri's analysis of the canonical ideal of an
  algebraic curve.
\newblock {\em Manuscripta Math.}, 61(2):223--248, 1988.
\newblock \href {https://doi.org/10.1007/BF01259331}
  {\path{doi:10.1007/BF01259331}}.

\end{thebibliography}
\end{document}